\documentclass[a4paper,11pt,leqno]{article}

\usepackage{amsmath}
\usepackage{amsthm}
\usepackage{amssymb}

\theoremstyle{plain}
\newtheorem{theorem}{Theorem}[section]
\newtheorem{lemma}[theorem]{Lemma}
\newtheorem{proposition}[theorem]{Proposition}
\newtheorem{corollary}[theorem]{Corollary}

\theoremstyle{definition}
\newtheorem{definition}[theorem]{Definition}

\theoremstyle{remark}
\newtheorem{remark}[theorem]{Remark}

\numberwithin{equation}{section}

\begin{document}

\title{\textbf{Superconductivity\\ and the BCS-Bogoliubov Theory}}

\author{Shuji Watanabe\\
Division of Mathematical Sciences\\
Graduate Schoool of Engineering, Gunma University\\
4-2 Aramaki-machi, Maebashi 371-8510, Japan\\
e-mail: watanabe@fs.aramaki.gunma-u.ac.jp}

\date{}

\maketitle

\begin{abstract}
\noindent First, we reformulate the BCS-Bogoliubov theory of superconductivity from the viewpoint of linear algebra. We define the BCS Hamiltonian on $\mathbb{C}^{2^{2M}}$, where $M$ is a positive integer. We discuss selfadjointness and symmetry of the BCS Hamiltonian as well as spontaneous symmetry breaking. Beginning with the gap equation, we give the well-known expression for the BCS state and find the existence of an energy gap. We also show that the BCS state has a lower energy than the normal state. Second, we introduce a new superconducting state explicitly and show from the viewpoint of linear algebra that this new state has a lower energy than the BCS state. Third, beginning with our new gap equation, we show from the viewpoint of linear algebra that we arrive at the results similar to those in the BCS-Bogoliubov theory.

\medskip

\noindent 2000 Mathematics Subject Classification: 15A90.

\medskip

\noindent Keywords and phrases: Superconductivity, the BCS-Bogoliubov theory, new superconducting state having a lower energy than the BCS state, new gap equation.

\end{abstract}

\section{Introduction}

Superconductivity is one of the historical landmarks in condensed matter physics. In 1911 Onnes found out the experimental fact that the electrical resistivity of mercury drops to zero below the temperature 4.2 K. Later the zero electrical resistivity is observed in many metals and alloys, and such a phenomenon is called superconductivity. The magnetic properties of superconductors as well as their electric properties are also astonishing. For example, the magnetic flux is excluded from the interior of a superconductor. This phenomenon was observed first by Meissner in 1933, and is called the Meissner effect. In 1957 Bardeen, Cooper and Schrieffer \cite{bcs} proposed the highly successful quantum theory, which is called the BCS theory. This theory is based on the idea that there is an attractive interaction between electrons caused by the phonons. Due to this interaction, the electrons form bound pairs. This superconducting state is called the BCS state and the Hamiltonian they dealt with is called the BCS Hamiltonian. Using a canonical transformation Bogoliubov \cite{bogoliubov} obtained the results similar to those in the BCS theory in 1958. This canonical transformation is called the Bogoliubov transformation, and is useful for studying the spectrum of the system of the electrons. This theory is called the Bogoliubov theory. The ground state of the BCS Hamiltonian is discussed by several authors. In 1961 Mattis and Lieb \cite{mattislieb} studied the wavefunction of the ground state of the BCS Hamiltonian under the condition that in the ground state, all the electrons in the neighborhood of the Fermi surface are paired. See Richardson \cite{richardson} and von Delft \cite{delft} for the ground state of the BCS Hamiltonian without the condition just above. From the viewpoint of $C^{\ast}$-algebra, Gerisch and Rieckers \cite{hubbard} studied a class of BCS-models to show that there is a unique $C^{\ast}$-dynamical system for each BCS-model.

Let $L,\, K_{max}>0$ be large enough and let us fix them. For $n_1,\, n_2,\, n_3\in \mathbb{Z}$, set
\[
\Lambda=\left\{ \frac{\,2\pi\,}{L}(n_1,\, n_2,\, n_3)\in\mathbb{R}^3:
\, \frac{\,2\pi\,}{L}\sqrt{n_1^2+n_2^2+n_3^2}\leq K_{max} \right\},
\]
Let the number of all the elements of $\Lambda$ be $M$ and let wave vector $k$ belong to $\Lambda$. The number $n_{k\sigma}$ of electrons with wave vector $k$ and spin $\sigma$ ($\sigma=\uparrow$ (spin up), $\downarrow$ (spin down)) is equal to 0 or 1, and so the number of the states
\begin{equation}\label{eq:state}
| \, n_{k\uparrow},\, n_{k\downarrow},\,n_{k'\uparrow},\,
 n_{k'\downarrow},\,\ldots \rangle,\qquad k,\,k',\,\ldots \in \Lambda
\end{equation}
is equal to $2^{2M}$. Here, $n_{k\uparrow}$,\  $n_{k\downarrow}=0,\,1$, and
the elements $k$ and $k'$ in \eqref{eq:state} are arranged in a certain order.
In condensed matter physics it is known that the norm of each state is equal
to 1 and that any two states are orthogonal to each other. We therefore choose
\[
\mathcal{H}=\mathbb{C}^{2^{2M}}
\]
as our Hilbert space $\mathcal{H}$, and denote each standard unit vector
in $\mathcal{H}=\mathbb{C}^{2^{2M}}$:
\[
e_i=(0,\,\ldots,\,0,\,\overset{i}{\widehat{1}},\,0,\,\ldots,\,0),\qquad
i=1,\,2,\,\ldots,\,2^{2M}
\]
by each state \eqref{eq:state} for simplicity. For example, we denote $e_1=(1,\,0,\,0,\ldots,\,0)$ and $e_2=(0,\,1,\,0,\ldots,\,0)$ by $| 0,\,0,\,0,\,\ldots \rangle$ and by $| 1,\,0,\,0,\,\ldots \rangle$, respectively. Moreover, we denote $e_{2^{2M}}=(0,\,0,\,\ldots,\,0,\,1)$ by $| 1,\,1,\,1,\,\ldots \rangle$. Here the symbol $| 0,\, 0,\,0,\,\ldots \rangle$ corresponds to the state $n_{k\uparrow}=n_{k\downarrow}=0$ for all $k\in \Lambda$, and $| 1,\, 0,\,0,\,\ldots \rangle$ to the state $n_{k\uparrow}=1$ and $n_{k\downarrow}=n_{k'\sigma}=0$ for all $k' \in \Lambda\setminus \{ k\}$ and for all $\sigma=\uparrow,\,\downarrow$. Moreover, $| 1,\,1,\,1,\,\ldots \rangle$ corresponds to the state $n_{k\uparrow}=n_{k\downarrow}=1$ for all $k\in \Lambda$. \quad We abbreviate $| 0,\, 0,\,0,\,\ldots \rangle$ to $| 0 \rangle$ and call it the vacuum vector in $\mathcal{H}=\mathbb{C}^{2^{2M}}$. We denote by $\left( \cdot\, ,\,\cdot\right)$ the inner product of $\mathcal{H}=\mathbb{C}^{2^{2M}}$.

The paper proceeds as follows. In section 2 we define the BCS Hamiltonian on $\mathcal{H}=\mathbb{C}^{2^{2M}}$. We discuss selfadjointness and symmetry of the BCS Hamiltonian as well as spontaneous symmetry breaking \cite{nambujona}. In section 3, beginning with the ``gap equation" (\cite{bcs}, \cite{bogoliubov}), we give the well-known expression for the BCS state and find the existence of an energy gap for excitation from the BCS state. In section 4 we obtain an expression for the energy difference between the BCS and normal states, and show that the BCS state has a lower energy than the normal state. In section 5 we introduce a new superconducting state explicitly and show that this new state has a lower energy than the BCS state, and hence than the normal state. On the basis of the results above we introduce a new gap equation in section 6. Beginning with our new gap equation we arrive at the results similar to those in section 3, i.e., to those in the BCS-Bogoliubov theory.

\section{The BCS Hamiltonian and spontaneous symmetry breaking}

In this section we define the BCS Hamiltonian on $\mathcal{H}=\mathbb{C}^{2^{2M}}$. We then discuss selfadjointness and symmetry of the BCS Hamiltonian as well as spontaneous symmetry breaking.

We assume that each creation operator and each annihilation operator depend only on wave vector $k\in\Lambda$ and on spin $\sigma$ of an electron. We denote the creation operator (resp. the annihilation operator) by $C_{k\sigma}^{\ast}$ (resp. by $C_{k\sigma}$). Note that $\displaystyle{ |\ldots,\, n_{k\uparrow},\, n_{k\downarrow},\,\ldots \rangle }$ \quad $(n_{k\uparrow},\, n_{k\downarrow}=0,\,1)$ stands for the corresponding standard unit vector in $\mathcal{H}=\mathbb{C}^{2^{2M}}$, as mentioned in the preceding section.

\begin{definition}
\[
\left\{ \begin{array}{ll} \displaystyle{
C_{k\uparrow}| \ldots,\, n_{k\uparrow},\, n_{k\downarrow},\,\ldots \rangle=
(-1)^{\sharp} \delta_{1,\,n_{k\uparrow}}| \ldots,\, n_{k\uparrow}-1,\,
 n_{k\downarrow},\,\ldots \rangle,} & \\
\noalign{ \vskip0.3cm }
\displaystyle{
C_{k\uparrow}^{\ast}| \ldots,\, n_{k\uparrow},\, n_{k\downarrow},\,\ldots
 \rangle=(-1)^{\sharp} \delta_{0,\,n_{k\uparrow}}| \ldots,\,
  n_{k\uparrow}+1,\, n_{k\downarrow},\,\ldots \rangle,} &
\end{array}\right.
\]
where the symbol $\sharp$ denotes the number of electrons arranged at the left of the symbol $n_{k\uparrow}$ above.
\[
\left\{ \begin{array}{ll} \displaystyle{
C_{k\downarrow}| \ldots,\, n_{k\uparrow},\, n_{k\downarrow},\,\ldots \rangle=
(-1)^{\sharp\sharp} \delta_{1,\,n_{k\downarrow}}| \ldots,\, n_{k\uparrow},\,
 n_{k\downarrow}-1,\,\ldots \rangle, } & \\
\noalign{ \vskip0.3cm }
\displaystyle{
C_{k\downarrow}^{\ast}| \ldots,\, n_{k\uparrow},\, n_{k\downarrow},\,\ldots
 \rangle=(-1)^{\sharp\sharp} \delta_{0,\,n_{k\downarrow}}| \ldots,\,
  n_{k\uparrow},\, n_{k\downarrow}+1,\,\ldots \rangle,} &
\end{array}\right.
\]
where the symbol $\sharp\sharp$ denotes the number of electrons arranged at the left of the symbol $n_{k\downarrow}$ above.
\end{definition}

On the basis of the definition we regard each of the creation and annihilation operators as a linear operator on $\mathcal{H}=\mathbb{C}^{2^{2M}}$. The definition immediately gives the following two lemmas.

\begin{lemma}\label{lm:blo}
The annihilation operator $C_{k\sigma}$ is a bounded linear operator on $\mathcal{H}=\mathbb{C}^{2^{2M}}$, and its adjoint operator coincides with the creation operator $C_{k\sigma}^{\ast}$. Moreover,
\[
C_{k\sigma}| 0 \rangle=0,\qquad C_{k\sigma}^{\ast}| \,1,\,1,\,1,\,\ldots \rangle=0,\qquad k\in \Lambda,\quad \sigma=\uparrow,\,\downarrow\,.
\]
\end{lemma}

\begin{lemma}\label{lm:cacr}
The operators $C_{k\sigma}$ and $C_{k\sigma}^{\ast}$ satisfy the canonical anticommutation relations on $\mathcal{H}=\mathbb{C}^{2^{2M}}$:
\[
\left\{ C_{k\sigma},\, C_{k'\sigma'}^{\ast} \right\}=\delta_{k k'}
\delta_{\sigma \sigma'},\qquad \left\{ C_{k\sigma},\, C_{k'\sigma'}
 \right\}=\left\{ C_{k\sigma}^{\ast},\, C_{k'\sigma'}^{\ast} \right\}=0,
\]
where $\{ A,\,B \}=AB+BA$.
\end{lemma}

\begin{remark}
\[
| \, n_{k\uparrow},\, n_{k\downarrow},\,n_{k'\uparrow},\, n_{k'\downarrow},\,\ldots \rangle=\left( C_{k\uparrow}^{\ast}\right)^{n_{k\uparrow}}\left( C_{k\downarrow}^{\ast}\right)^{n_{k\downarrow}}\left( C_{k'\uparrow}^{\ast}\right)^{n_{k'\uparrow}}\left( C_{k'\downarrow}^{\ast}\right)^{n_{k'\downarrow}} \,\cdots | 0 \rangle.
\]
\end{remark}

Let $m$ and $\mu$ stand for the electron mass and the chemical potential, respectively. Here, $m,\,\mu>0$. Set $\xi_k=\hslash^2|k|^2/(2m)-\mu$. The BCS Hamiltonian \cite{bcs} is given by
\begin{equation}\label{eq:hamiltonian}
H=\sum_{k\in\Lambda,\,\sigma=\uparrow,\,\downarrow}\xi_k\,
 C_{k\,\sigma}^{\ast}C_{k\,\sigma}+\sum_{k,\,k'\in\Lambda}U_{k,\,k'}
 C_{k'\uparrow}^{\ast}C_{-k'\downarrow}^{\ast}C_{-k\downarrow}
 C_{k\uparrow}\,.
\end{equation}
Here, $U_{k,\,k'}$ is a function of $k$ and $k'$, and satisfies $U_{k,\,k'}\leq 0$, $U_{k',\,k}=U_{k,\,k'}$, $U_{-k,\,-k'}=U_{k,\,k'}$ and $U_{k,\,k}=0$.

Lemma \ref{lm:blo} immediately yields the following.

\begin{proposition}
The BCS Hamiltonian $H$ is a bounded, selfadjoint operator on $\mathcal{H}=\mathbb{C}^{2^{2M}}$.
\end{proposition}

The bounded, selfadjoint operator
\begin{equation}\label{eq:gnrtr}
G=\sum_{k\in\Lambda,\,\sigma=\uparrow,\,\downarrow} C_{k\,\sigma}^{\ast}
C_{k\,\sigma}
\end{equation}
generates a strongly continuous unitary group $\displaystyle{\left\{ e^{i\,\alpha\,G}\right\}}_{\alpha\in\mathbb{R}}$ on $\mathcal{H}=\mathbb{C}^{2^{2M}}$. As is shown in Proposition \ref{prp:phase}, the transformation $\displaystyle{e^{i\,\alpha\,G}}$ gives rise to a phase transformation of the creation (the annihilation) operator.

A straightforward calculation gives the following.

\begin{lemma}\label{lm:formula}
Let $A$ be a bounded linear operator on a Hilbert space $\mathcal{H}$ and $B$ a bounded, selfadjoint operator on $\mathcal{H}$. Let $\alpha\in\mathbb{R}$. Then, for $f\in\mathcal{H}$,
\begin{equation*}
\sum_{n=0}^N \frac{\, \left( i\alpha\right)^n\,}{n!}
\overbrace{ [\,[\,\ldots\, [\, [ }^{n} \,A,\,\overbrace{B\,],\,B\,], \ldots,\,B\,],\, B\,]}^{n}\, f \longrightarrow e^{-i\,\alpha\,B}A\, e^{i\,\alpha\,B}f
\qquad \mbox{in} \;\;\mathcal{H}.
\end{equation*}
Here, $[A,\, B]=AB-BA$.
\end{lemma}

\begin{proposition}\label{prp:phase} Let $G$ be as in \eqref{eq:gnrtr} and
$H$ as in \eqref{eq:hamiltonian}. Then, for $\alpha\in\mathbb{R}$, 
\[
e^{-i\,\alpha\,G}C_{k\,\sigma}\, e^{i\,\alpha\,G}=e^{i\,\alpha}C_{k\,\sigma}\,,
\qquad e^{-i\,\alpha\,G}C_{k\,\sigma}^{\ast}\, e^{i\,\alpha\,G}=e^{-i\,\alpha}
C_{k\,\sigma}^{\ast}\,.
\]
Consequently, \quad $\displaystyle{ e^{-i\,\alpha\,G}H\, e^{i\,\alpha\,G}=H}$.
\end{proposition}

\begin{proof}
By Lemma \ref{lm:cacr}, $\displaystyle{
\left[C_{k\,\sigma},\, G\right]=\sum_{k'\in\Lambda,\,\sigma'=\uparrow,\,\downarrow} \left[C_{k\,\sigma},\, C_{k'\,\sigma'}^{\ast}C_{k'\,\sigma'}\right]=C_{k\,\sigma} }$. The result thus follows from Lemma \ref{lm:formula}.
\end{proof}

\begin{remark}
In other words, the equality $[G,\, H]=0$ holds on $\mathcal{H}=\mathbb{C}^{2^{2M}}$, as is checked directly. Proposition \ref{prp:phase} implies that the transformation $e^{i\,\alpha\,G}$ leaves the BCS Hamiltonian $H$ invariant. In this case the BCS Hamiltonian $H$ is said to have global U(1) symmetry.
\end{remark}

\begin{definition}\label{dfn:ssb}
Let $G$ be as in \eqref{eq:gnrtr}. Suppose that there is the ground state $\Psi_0\in\mathcal{H}=\mathbb{C}^{2^{2M}}$ of the BCS Hamiltonian $H$. The global U(1) symmetry is said to be spontaneously broken if there is a bounded linear operator $A$ on $\mathcal{H}=\mathbb{C}^{2^{2M}}$ satisfying
\[
\left( \Psi_0,\, \left[ \, G,\, A\right]\Psi_0\right)\not=0.
\]
\end{definition}

Lemma \ref{lm:cacr} immediately gives the following.

\begin{lemma}\label{lm:ssblemma} Set $A=C_{-k\downarrow}C_{k\uparrow}$ in Definition
\ref{dfn:ssb}. Then
\[
\left( \Psi_0,\, \left[ \, G,\, C_{-k\downarrow}C_{k\uparrow} \right]\Psi_0\right)
=-2\left( \Psi_0,\, C_{-k\downarrow}C_{k\uparrow}\Psi_0\right).
\]
\end{lemma}

\begin{remark}\label{rmk:ssbremark}
If $\displaystyle{\left( \Psi_0,\, C_{-k\downarrow}C_{k\uparrow}\Psi_0\right)
\not=0}$, then the global U(1) symmetry is spontaneously broken.
\end{remark}

\begin{remark}
The concept of spontaneous symmetry breaking was introduced first by Nambu and Jona-Lasinio \cite{nambujona} in 1961. This plays an important role in quantum mechanics such as the BCS-Bogoliubov theory and quantum gauge field theory.
\end{remark}

\section{An energy gap for excitation from the BCS state}

In this section we give the well-known expression for the BCS state and find the existence of an energy gap for excitation from the BCS state.

Let $\Delta_k$ and $\theta_k$ be functions of $k\in\Lambda$. We assume the existence of the following $\Delta_k$ :\quad $\Delta_k$ satisfies $\Delta_{-k}=
\Delta_k$ and is a solution to the ``gap equation" (\cite{bcs}, \cite{bogoliubov})
\begin{equation}\label{eq:gapequation}
\Delta_k=-\frac{1}{\,2\,}\sum_{k'\in\Lambda}U_{k,\,k'}\,
\frac{\Delta_{k'}}{\,\sqrt{\,\xi_{k'}^2+\Delta_{k'}^2\,}\,}\,.
\end{equation}
Note that if $\Delta_k$ is a solution to the gap equation, then
$-\Delta_k$ is also a solution. So we let $\Delta_k \geq 0$. Let
$\theta_k$ satisfy (\cite{bcs}, \cite{bogoliubov})
\begin{equation}\label{eq:theta}
\sin 2\theta_k=\frac{\Delta_k}{\,\sqrt{\,\xi_k^2+\Delta_k^2\,}\,}\,,\qquad
\cos 2\theta_k=\frac{\xi_k}{\,\sqrt{\,\xi_k^2+\Delta_k^2\,}\,}
\end{equation}
with $0 \leq \theta_k \leq \pi/2$. Note that $\theta_{-k}=\theta_k$ . We denote by $G_{B}$ the following bounded, selfadjoint operator on $\mathcal{H}=\mathbb{C}^{2^{2M}}$:
\[
G_{B}=i\sum_{k\in\Lambda} \theta_k \left( C_{-k\downarrow}C_{k\uparrow}
-C_{k\uparrow}^{\ast}C_{-k\downarrow}^{\ast} \right).
\]
Here, $\theta_k$ is given by \eqref{eq:theta}. We set
\[
\Psi_{BCS}=e^{iG_B}|0\rangle\in\mathcal{H}=\mathbb{C}^{2^{2M}}
\]
and call it the BCS state (\cite{bcs}, \cite{bogoliubov}). The BCS state is explicitly expressed as follows.

\begin{lemma}\label{lm:bcsstate} \quad $\displaystyle{
\Psi_{BCS}=\left\{ \,\prod_{k\in\Lambda}\left( \cos\theta_k+\sin\theta_k\,C_{k\uparrow}^{\ast}C_{-k\downarrow}^{\ast} \right)\,\right\}|0\rangle
}$.
\end{lemma}

\begin{proof}
Since \  $\displaystyle{ 
\left[\, C_{-k\downarrow}C_{k\uparrow}-C_{k\uparrow}^{\ast}C_{-k\downarrow}^{\ast},\, C_{-k'\downarrow}C_{k'\uparrow}-C_{k'\uparrow}^{\ast}C_{-k'\downarrow}^{\ast}  \,\right]=0}$,\  $k\not=k'$, \  it follows
\[
\Psi_{BCS}=\left\{ \,\prod_{k\in\Lambda}\exp\left[ -\theta_k \left(
C_{-k\downarrow}C_{k\uparrow}-C_{k\uparrow}^{\ast}C_{-k\downarrow}^{\ast}
\right)\right]\,\right\}|0\rangle.
\]
A straightforward calculation based on Lemmas \ref{lm:blo} and \ref{lm:cacr} gives that for $n=1,\,2,\,3,\,\ldots$,
\begin{eqnarray*}
\left[ -\theta_k \left( C_{-k\downarrow}C_{k\uparrow}-C_{k\uparrow}^{\ast}C_{-k\downarrow}^{\ast} \right) \right]^{2n-1} |0\rangle&=&
(-1)^{n-1}\,\theta_k^{\,2n-1}\,C_{k\uparrow}^{\ast}C_{-k\downarrow}^{\ast}\, |0\rangle,\\
\left[ -\theta_k \left( C_{-k\downarrow}C_{k\uparrow}-C_{k\uparrow}^{\ast}C_{-k\downarrow}^{\ast} \right) \right]^{2n} |0\rangle&=&
(-1)^n\,\theta_k^{\,2n}\, |0\rangle.
\end{eqnarray*}
Hence,\quad $\displaystyle{
\exp\left[ -\theta_k \left( C_{-k\downarrow}C_{k\uparrow}
-C_{k\uparrow}^{\ast}C_{-k\downarrow}^{\ast}\right)\right] \,|0\rangle=
\left( \cos\theta_k+\sin\theta_k\,C_{k\uparrow}^{\ast}C_{-k\downarrow}^{\ast}
\right)\,|0\rangle}$.\\ The result thus follows.
\end{proof}

\begin{remark}
In 1957 Bardeen, Cooper and Schrieffer \cite{bcs} introduced the well-known expression in Lemma \ref{lm:bcsstate}.
\end{remark}

\begin{corollary}\label{crl:bcsexpectationvalue}\hfill

\noindent {\rm (a)}\quad $\displaystyle{ \left(\Psi_{BCS},\, C_{-k\downarrow}C_{k\uparrow}\Psi_{BCS}\right)=\left(\Psi_{BCS},\,C_{k\uparrow}^{\ast}C_{-k\downarrow}^{\ast}\Psi_{BCS}\right)=\frac{1}{\,2\,}\sin 2\theta_k }$ .

\noindent {\rm (b)}\quad $\displaystyle{ \Delta_k=-\sum_{k'\in\Lambda}
U_{k,\,k'}\, \left(\Psi_{BCS},\, C_{-k'\downarrow}C_{k'\uparrow}\Psi_{BCS}\right) }$.
\end{corollary}

\begin{proof} \quad Combining Lemma \ref{lm:bcsstate} with Lemmas \ref{lm:blo} and \ref{lm:cacr} gives (a). Part (b) follows immediately from (a) and \eqref{eq:gapequation}.
\end{proof}

Let us recall Remark \ref{rmk:ssbremark}. We replace $\Psi_0$ in Remark \ref{rmk:ssbremark} by $\Psi_{BCS}$ and set (for all $k\in\Lambda$)
\begin{equation}\label{eq:vev}
\left\{ \begin{array}{ll} \displaystyle{
C_{-k\downarrow}C_{k\uparrow}=
\left(\Psi_{BCS},\, C_{-k\downarrow}C_{k\uparrow}\Psi_{BCS}\right)+b_k \,,} & \\
\noalign{ \vskip0.3cm }
\displaystyle{
C_{k\uparrow}^{\ast}C_{-k\downarrow}^{\ast}=
\left(\Psi_{BCS},\,C_{k\uparrow}^{\ast}C_{-k\downarrow}^{\ast}\Psi_{BCS}\right)+b_k^{\ast}
\,.} &
\end{array}\right.
\end{equation}
A straightforward calculation gives the following.

\begin{lemma}\label{lm:h-hm} Set
\begin{eqnarray*}
H_M &=&\sum_{k\in\Lambda,\,\sigma=\uparrow,\,\downarrow}\xi_k\, 
C_{k\,\sigma}^{\ast}C_{k\,\sigma}-\sum_{k\in\Lambda}\Delta_k\left(
 C_{-k\downarrow}C_{k\uparrow}+C_{k\uparrow}^{\ast}C_{-k\downarrow}^{\ast}
 \right) \\
& &+\sum_{k\in\Lambda}\Delta_k\left(\Psi_{BCS},\, C_{-k\downarrow}
C_{k\uparrow}\Psi_{BCS}\right).
\end{eqnarray*}
Then the BCS Hamiltonian \eqref{eq:hamiltonian} is rewritten as
$\displaystyle{
H=H_M+\sum_{k,\,k'\in\Lambda}U_{k,\,k'}\, b_{k'}^{\ast}\, b_k }$ .
\end{lemma}

\begin{remark}
The Hamiltonian $H_M$ is called the mean field approximation for the BCS Hamiltonian $H$.
\end{remark}

\begin{lemma}\label{lm:hkvk} Let $\displaystyle{h_k=C_{k\,\uparrow}^{\ast}C_{k\,\uparrow}+C_{-k\,\downarrow}^{\ast}C_{-k\,\downarrow}}$ and let $\displaystyle{v_k=C_{-k\downarrow}C_{k\uparrow}+C_{k\uparrow}^{\ast}C_{-k\downarrow}^{\ast}}$.

\noindent {\rm (a)}\quad For $n=1,\,2,\,3,\,\ldots$,
\begin{eqnarray*}
\overbrace{ [\,\ldots\,[ }^{2n-1} \,\xi_kh_k-\Delta_kv_k,\,\overbrace{iG_B\,],\,\ldots,\, iG_B\,]}^{2n-1}&=&(-1)^{n-1}\left( 2\theta_k \right)^{2n-1}\times\\
& &\quad \times \left\{ \xi_kv_k+\Delta_k\left( h_k-1 \right) \right\},\\
\overbrace{ [\,\ldots\, [ }^{2n} \,\xi_kh_k-\Delta_kv_k,\,\overbrace{iG_B\,],\,\ldots,\, iG_B\,]}^{2n}&=&(-1)^n\left( 2\theta_k \right)^{2n}\times\\
& &\quad \times \left\{ \xi_k\left( h_k-1 \right)-\Delta_kv_k \right\}.
\end{eqnarray*}

\noindent {\rm (b)} 
\begin{eqnarray*}
e^{-iG_B}\left(\xi_kh_k-\Delta_kv_k\right)e^{iG_B}
&=&\left( \xi_k\cos 2\theta_k+\Delta_k\sin 2\theta_k \right) h_k \\
& &+\left( \xi_k\sin 2\theta_k-\Delta_k\cos 2\theta_k \right) v_k \\
& &+\left( 2\xi_k\sin^2\theta_k-\Delta_k\sin 2\theta_k \right).
\end{eqnarray*}
\end{lemma}

\begin{proof} (a) \quad A straightforward calculation based on Lemma \ref{lm:cacr} gives
\begin{equation}\label{eq:cr}
[\, h_k,\,iG_B \,]=2\theta_k v_k\,,\qquad [\, v_k,\,iG_B \,]=-2\theta_k\left( h_k-1\right),
\end{equation}
and hence the result is true for $n=1$. Suppose that the result is true for $n$. Then, by \eqref{eq:cr},
\begin{eqnarray*}
& &\overbrace{ [\,\ldots\, [ }^{2n+1} \,\xi_kh_k-\Delta_kv_k,\,\overbrace{iG_B\,],\,\ldots,\, iG_B\,]}^{2n+1} \\
&=&(-1)^n\left( 2\theta_k \right)^{2n}
[\,\xi_k\left( h_k-1 \right)-\Delta_kv_k,\, iG_B\,] \\
&=&(-1)^n\left( 2\theta_k \right)^{2n+1}
\left\{ \xi_kv_k+\Delta_k\left( h_k-1 \right) \right\},
\end{eqnarray*}
and hence
\begin{eqnarray*}
& &\overbrace{ [\,\ldots\, [ }^{2n+2} \,\xi_kh_k-\Delta_kv_k,\,\overbrace{iG_B\,],\,\ldots,\, iG_B\,]}^{2n+2} \\
&=&(-1)^n\left( 2\theta_k \right)^{2n+1}
[\,\xi_kv_k+\Delta_k\left( h_k-1 \right),\, iG_B\,] \\
&=&(-1)^{n+1}\left( 2\theta_k \right)^{2(n+1)}
\left\{ \xi_k\left( h_k-1 \right)-\Delta_kv_k \right\}.
\end{eqnarray*}
Therefore the result is true for $n+1$, and hence for every $n=1,\,2,\,3,\,\ldots$.

Part (b) follows immediately from (a).
\end{proof}

We employ the well-known Bogoliubov transformation of $C_{k\,\sigma}$ \cite{bogoliubov} :
\begin{equation}\label{eq:gamma}
\gamma_{k\,\sigma}=e^{iG_B}C_{k\,\sigma}e^{-iG_B}.
\end{equation}
Note that the operator $\gamma_{k\,\sigma}$ and its adjoint operator $\gamma_{k\,\sigma}^{\ast}$ are both bounded linear operators on $\mathcal{H}=\mathbb{C}^{2^{2M}}$.

\begin{proposition}\label{prp:hm}
\begin{eqnarray*}
H_M&=&\sum_{k\in\Lambda,\,\sigma=\uparrow,\,\downarrow}
\sqrt{\,\xi_k^2+\Delta_k^2\,}\,\gamma_{k\,\sigma}^{\ast}\gamma_{k\,\sigma}\\
& &+\sum_{k\in\Lambda}\left\{ \xi_k-\sqrt{\,\xi_k^2+\Delta_k^2\,}
+\Delta_k\left(\Psi_{BCS},\, C_{-k\downarrow}C_{k\uparrow}\Psi_{BCS}\right)\right\}.
\end{eqnarray*}
\end{proposition}

\begin{proof}
Combining Lemma \ref{lm:hkvk} (b) with \eqref{eq:theta} gives
\begin{eqnarray*}
& &e^{-iG_B}H_Me^{iG_B} \\
&=&\sum_{k\in\Lambda} e^{-iG_B}\left(\xi_kh_k-\Delta_kv_k\right)e^{iG_B}+\sum_{k\in\Lambda}\Delta_k\left(\Psi_{BCS},\, C_{-k\downarrow}C_{k\uparrow}\Psi_{BCS}\right)\\
&=&\sum_{k\in\Lambda} \left\{ \sqrt{\,\xi_k^2+\Delta_k^2\,}\,h_k+\xi_k-\sqrt{\,\xi_k^2+\Delta_k^2\,}+\Delta_k\left(\Psi_{BCS},\, C_{-k\downarrow}C_{k\uparrow}\Psi_{BCS}\right)\right\}.
\end{eqnarray*}
The result thus follows from \eqref{eq:gamma}.
\end{proof}

Proposition \ref{prp:hm} immediately yields the following.

\begin{corollary}\label{crl:ground} {\rm (a)}\quad The BCS state $\Psi_{BCS}$ is the ground state of $H_M$, and the ground state energy $E_{BCS}$ is given by
\[
E_{BCS}=\sum_{k\in\Lambda}\left\{ \xi_k-\sqrt{\,\xi_k^2+\Delta_k^2\,}
+\Delta_k\left(\Psi_{BCS},\, C_{-k\downarrow}C_{k\uparrow}\Psi_{BCS}\right)\right\}.
\]

\noindent {\rm (b)}\quad Let $E_{BCS}$ be as in (a). Then the spectrum of $H_M$ is given by
\[
\sigma\left( H_M\right)=\left\{ \sum_{k\in\Lambda} \sqrt{\xi_k^2+\Delta_k^2}\,
\left(\, N_{k\uparrow}+N_{k\downarrow}\,\right)+E_{BCS} \right\}_{N_{k\uparrow},\, N_{k\downarrow}=0,\,1}.
\]
\end{corollary}

\begin{remark}
Corollary \ref{crl:ground} (b) implies that it takes a finite energy $\sqrt{\xi_k^2+\Delta_k^2}$ ($>\Delta_k$) to excite a particle from the BCS state to an upper energy state. So the function $\Delta_k$ of $k\in\Lambda$ corresponds exactly to the energy gap, and hence $\Delta_k$ is called the gap function (see Bardeen, Cooper and Schreiffer \cite{bcs}, and Bogoliubov \cite{bogoliubov}).
\end{remark}

We now study some properties of the operators $\gamma_{k\,\sigma}$ in \eqref{eq:gamma} (see Bogoliubov \cite{bogoliubov}).

\begin{corollary}\label{crl:properties} The operators $\gamma_{k\,\sigma}$ and $\gamma_{k\,\sigma}^{\ast}$ satisfy the following.

\noindent {\rm (a)} \quad $\left\{ \gamma_{k\sigma},\, \gamma_{k'\sigma'}^{\ast} \right\}=\delta_{k k'}\delta_{\sigma \sigma'},\qquad \left\{ \gamma_{k\sigma},\, \gamma_{k'\sigma'} \right\}=\left\{ \gamma_{k\sigma}^{\ast},\, \gamma_{k'\sigma'}^{\ast} \right\}=0$.

\noindent {\rm (b)}\quad $\gamma_{k\,\sigma}\Psi_{BCS}=0$\quad for each $k\in\Lambda$ and for each $\sigma=\uparrow,\,\downarrow$ .

\noindent {\rm (c)}\quad $\displaystyle{
\left\{ \begin{array}{ll} \displaystyle{
\gamma_{k\uparrow}=\cos\theta_k\,C_{k\uparrow}
-\sin\theta_k\, C_{-k\downarrow}^{\ast}\,,}
& \\ \noalign{\vskip0.3cm}
\displaystyle{
\gamma_{-k\downarrow}=\sin\theta_k\,C_{k\uparrow}^{\ast}
+\cos\theta_k\,C_{-k\downarrow}\,.} &
\end{array}
\right.
}$

\noindent {\rm (d)}\quad $\displaystyle{
\left\{ \begin{array}{ll} \displaystyle{
C_{k\uparrow}=\cos\theta_k\,\gamma_{k\uparrow}
+\sin\theta_k\, \gamma_{-k\downarrow}^{\ast}\,,}
& \\ \noalign{\vskip0.3cm}
\displaystyle{
C_{-k\downarrow}=-\sin\theta_k\,\gamma_{k\uparrow}^{\ast}
+\cos\theta_k\,\gamma_{-k\downarrow}\,.} &
\end{array}
\right.
}$
\end{corollary}

\begin{proof}
Part (a) follows immediately from Lemma \ref{lm:cacr} and \eqref{eq:gamma}.\quad The equality $C_{k\sigma}| 0 \rangle=0$ in Lemma \ref{lm:blo} yields (b).

\noindent (c)\quad  A straightforward calculation based on Lemma \ref{lm:cacr} gives
\begin{eqnarray*}
\overbrace{ [\,\ldots\,[ }^{2n-1} \,C_{k\uparrow},\,\overbrace{iG_B\,],\,\ldots,\, iG_B\,]}^{2n-1}&=&(-1)^{n-1} \theta_k^{\,2n-1}C_{-k\downarrow}^{\ast},\\
\overbrace{ [\,\ldots\, [ }^{2n} \,C_{k\uparrow},\,\overbrace{iG_B\,],\,\ldots,\, iG_B\,]}^{2n}&=&(-1)^n \theta_k^{\,2n}C_{k\uparrow}
\end{eqnarray*}
for $n=1,\,2,\,3,\,\ldots$. The first equality thus follows from Lemma \ref{lm:formula}. Similarly,
\begin{eqnarray*}
\overbrace{ [\,\ldots\,[ }^{2n-1} \,C_{-k\downarrow},\,\overbrace{iG_B\,],\,\ldots,\, iG_B\,]}^{2n-1}&=&(-1)^n \theta_k^{\,2n-1}C_{k\uparrow}^{\ast},\\
\overbrace{ [\,\ldots\, [ }^{2n} \,C_{-k\downarrow},\,\overbrace{iG_B\,],\,\ldots,\, iG_B\,]}^{2n}&=&(-1)^n \theta_k^{\,2n}C_{-k\downarrow}.
\end{eqnarray*}
The second equality follows in a similar manner.

Part (d) follows immediately from (c).
\end{proof}

\section{The energy difference between the BCS and normal states}

In this section we obtain an expression for the energy difference between the BCS and normal states, and show that the BCS state has a lower energy than the normal state.

Let $\Delta_k=0$ for all $k\in\Lambda$. Then, by \eqref{eq:theta}, $\sin2\theta_k=0$ and $\cos2\theta_k=-1$ for $k\in\Lambda$ satisfying $\xi_k<0$. Hence, $\theta_k=\pi/2$. On the other hand, $\sin2\theta_k=0$ and $\cos2\theta_k=1$ for $k\in\Lambda$ satisfying $\xi_k>0$. Hence, $\theta_k=0$. When $\Delta_k=0$ for all $k\in\Lambda$, we assume that $\theta_k=\pi/2$ for $k\in\Lambda$ satisfying $\xi_k=0$. Therefore, if $\Delta_k=0$ for all $k\in\Lambda$, then the BCS state $\Psi_{BCS}$ coincides with the ``Fermi vacuum" $\Psi_F\in\mathcal{H}=\mathbb{C}^{2^{2M}}$ by Lemma \ref{lm:bcsstate} (see \cite{bcs} and \cite{bogoliubov}). Here the Fermi vacuum $\Psi_F$ corresponds to the normal state and is defined by
\[
\Psi_F=\left\{ \,\prod_{k\; (\xi_k\leq 0)} C_{k\uparrow}^{\ast}C_{-k\downarrow}^{\ast} \right\} |0\rangle,
\]
where the symbol $k\; (\xi_k\leq 0)$ stands for $k\in\Lambda$ satisfying $\xi_k\leq 0$.

\begin{lemma}\quad Let $E_{BCS}$ be as in Corollary \ref{crl:ground} (a).

\noindent {\rm (a)}
\begin{eqnarray*}
\left(\Psi_{BCS},\, H\Psi_{BCS}\right)&=&\sum_{k\in\Lambda}\left(\xi_k-\frac{\xi_k^2}{\,\sqrt{\,\xi_k^2+\Delta_k^2\,}\,}\right) \\
& &+\frac{1}{\,4\,}\sum_{k,\,k'\in\Lambda}U_{k,\,k'}\,\frac{\Delta_k\,\Delta_{k'}}{\,\sqrt{\,\xi_k^2+\Delta_k^2\,}\sqrt{\,\xi_{k'}^2+\Delta_{k'}^2\,} \,} \\
&=& E_{BCS}\,.
\end{eqnarray*}

\noindent {\rm (b)}\quad $\displaystyle{\left( \Psi_F,\, H\Psi_F\right)=\sum_{k\in\Lambda}
\left( \xi_k-\left| \xi_k \right| \right) }$.
\end{lemma}

\begin{proof} (a)
\begin{eqnarray*}
\left(\Psi_{BCS},\, H\Psi_{BCS}\right)&=&\sum_{k\in\Lambda,\,\sigma=\uparrow,\,\downarrow}
\xi_k \left(\Psi_{BCS},\, C_{k\,\sigma}^{\ast}C_{k\,\sigma}\Psi_{BCS}\right)\\
& &\quad +\sum_{k,\,k'\in\Lambda}U_{k,\,k'}\left(\Psi_{BCS},\, 
C_{k'\uparrow}^{\ast}C_{-k'\downarrow}^{\ast}C_{-k\downarrow} C_{k\uparrow}
\Psi_{BCS}\right).
\end{eqnarray*}
A straightforward calculation based on Lemmas \ref{lm:blo} and \ref{lm:cacr}
gives
\begin{eqnarray*}
& &\left(\Psi_{BCS},\, C_{k\,\sigma}^{\ast}C_{k\,\sigma}\Psi_{BCS}\right)=
\sin^2\theta_k\,,\\
& &\left(\Psi_{BCS},\,C_{k'\uparrow}^{\ast}C_{-k'\downarrow}^{\ast}C_{-k\downarrow} C_{k\uparrow}\Psi_{BCS}\right)=\frac{1}{\,4\,}\sin 2\theta_k\,\sin 2\theta_{k'}\,.
\end{eqnarray*}
Part (a) thus follows from \eqref{eq:gapequation}, \eqref{eq:theta} and Corollary \ref{crl:bcsexpectationvalue} (a). Part (b) follows immediately from (a). 
\end{proof}

Combining this lemma with \eqref{eq:gapequation} immediately yields the following.

\begin{proposition} The BCS state $\Psi_{BCS}$ has a lower energy than the
Fermi vacuum $\Psi_F$ (the normal state), i.e.,
\[
\left( \Psi_{BCS},\, H\Psi_{BCS}\right)-\left( \Psi_F,\, H\Psi_F \right)=-\frac{1}{\,2\,}
\sum_{k\in\Lambda} \frac{\,\left( \sqrt{\,\xi_k^2+\Delta_k^2\,}-\left| \xi_k\right|\right)^2\,}{\sqrt{\,\xi_k^2+\Delta_k^2\,}}<0\,.
\]
\end{proposition}

\section{A new superconducting state having a lower energy than the BCS state}

In this section we introduce a new superconducting state explicitly and show that this new state has a lower energy than the BCS state, and hence than the normal state.

Set $E_k=\sqrt{\,\xi_k^2+\Delta_k^2\,}$, $k\in\Lambda$ and set $B_k=C_{-k\downarrow}C_{k\uparrow}$. We abbreviate $\sin\theta_k$ (resp. $\cos\theta_k$) to $S_k$ (resp. to $C_k$). We consider the following vector in $\mathcal{H}=\mathbb{C}^{2^{2M}}$:
\begin{equation*}
\Psi=\frac{\Psi_{BCS}+\Phi}{\,\sqrt{\, 1+\left( \Phi,\,\Phi \right) \,}\,},
\end{equation*}
where \quad $\displaystyle{
\Phi=\frac{1}{\,2\,}\sum_{p,\,p'\in\Lambda} \frac{\,U_{p,\,p'}
\left( C_p^2\,S_{p'}^2+C_{p'}^2\,S_p^2\right)\,}{E_p+E_{p'}}\,
\gamma_{p\uparrow}^{\ast}\,\gamma_{-p\downarrow}^{\ast}\,
\gamma_{p'\uparrow}^{\ast}\,\gamma_{-p'\downarrow}^{\ast}\,\Psi_{BCS}}$ .

We prepare some lemmas.

\begin{lemma}\label{lm:ebcs} {\rm (a)}\quad $\displaystyle{\left( \Psi_{BCS},\,\Phi \right)=0}$.

\noindent {\rm (b)}
\[
H_M\Phi=E_{BCS}\Phi+2\sum_{p,\,p'\in\Lambda} \frac{\,E_{p'}U_{p,\,p'}\left( C_p^2\,S_{p'}^2+C_{p'}^2\,S_p^2\right)\,}{E_p+E_{p'}}\,
\gamma_{p\uparrow}^{\ast}\,\gamma_{-p\downarrow}^{\ast}\,
\gamma_{p'\uparrow}^{\ast}\,\gamma_{-p'\downarrow}^{\ast}\,\Psi_{BCS}.
\]

\noindent {\rm (c)}\quad $\displaystyle{\left( \Psi,\,H_M\Psi \right)=E_{BCS}
+\frac{1}{\,1+\left( \Phi,\,\Phi \right)\,}
\sum_{p,\,p'\in\Lambda} \frac{\,U_{p,\,p'}^{\,2}
\left( C_p^2\,S_{p'}^2+C_{p'}^2\,S_p^2\right)^2\,}{E_p+E_{p'}}}$.
\end{lemma}

\begin{proof}
Part (a) follows immediately from Corollary \ref{crl:properties} (b).

\noindent (b)\quad By Proposition \ref{prp:hm},\quad $\displaystyle{
H_M\Phi=E_{BCS}\Phi+\sum_{k\in\Lambda} E_k\left(
\gamma_{k\,\uparrow}^{\ast}\gamma_{k\,\uparrow}
+\gamma_{-k\,\downarrow}^{\ast}\gamma_{-k\,\downarrow} \right)\Phi}$ .

\noindent A straightforward calculation based on Corollary \ref{crl:properties} (a) gives
\[
\gamma_{k\,\uparrow}^{\ast}\gamma_{k\,\uparrow}\Phi
=\gamma_{-k\,\downarrow}^{\ast}\gamma_{-k\,\downarrow}\Phi
=\sum_{p\in\Lambda} \frac{\,U_{p,\,k}\left( C_p^2\,S_k^2+C_k^2\,S_p^2\right)\,}{E_p+E_k}\, \gamma_{p\uparrow}^{\ast}\,\gamma_{-p\downarrow}^{\ast}\,
\gamma_{k\uparrow}^{\ast}\,\gamma_{-k\downarrow}^{\ast}\,\Psi_{BCS}\,,
\]
from which (b) follows.

\noindent (c)\quad By (a) and (b),
\begin{eqnarray*}
\left( \Psi,\,H_M\Psi \right)&=&\frac{1}{\,1+\left( \Phi,\,\Phi \right)\,}
\left\{\, E_{BCS}+\left( \Phi,\,H_M\Phi \right)\,\right\} \\
&=&E_{BCS}+\frac{2}{\,1+\left( \Phi,\,\Phi \right)\,}
\sum_{p,\,p'\in\Lambda} \frac{\,E_{p'}U_{p,\,p'}\left( C_p^2\,S_{p'}^2
+C_{p'}^2\,S_p^2\right)\,}{E_p+E_{p'}}\times \\
& &\qquad\qquad \times\left( \Phi,\,
\gamma_{p\uparrow}^{\ast}\,\gamma_{-p\downarrow}^{\ast}\,
\gamma_{p'\uparrow}^{\ast}\,\gamma_{-p'\downarrow}^{\ast}\,\Psi_{BCS}\right).
\end{eqnarray*}
Part (c) thus follows from Corollary \ref{crl:properties} (a).
\end{proof}

Set $H'=H-H_M$ . Then, by Lemma \ref{lm:h-hm},
\begin{equation}\label{eq:hprime}
H'=\sum_{k,\,k'\in\Lambda}U_{k,\,k'}\left\{ B_{k'}^{\ast}\,B_k
-C_{k'}\,S_{k'}\left( B_k^{\ast}+B_k \right)+C_k\,S_k\,C_{k'}\,S_{k'}\right\}.
\end{equation}

\begin{lemma}\label{lm:hprimeebcs} Let $H'$ be as in \eqref{eq:hprime}.

\noindent {\rm (a)}\quad $\displaystyle{
H'\Psi_{BCS}=-\sum_{k,\,k'\in\Lambda}U_{k,\,k'}\,S_k^2\,C_{k'}^2\,
\gamma_{k\uparrow}^{\ast}\,\gamma_{-k\downarrow}^{\ast}\,
\gamma_{k'\uparrow}^{\ast}\,\gamma_{-k'\downarrow}^{\ast}\,\Psi_{BCS}
}$ .

\noindent {\rm (b)}\quad $\displaystyle{\left( \Psi_{BCS},\,H'\Psi_{BCS} \right)=0}$.

\noindent {\rm (c)}\quad $\displaystyle{\left( \Phi,\,H'\Psi_{BCS} \right)=
-\frac{1}{\,2\,}\sum_{p,\,p'\in\Lambda} \frac{\,U_{p,\,p'}^{\,2}
\left( C_p^2\,S_{p'}^2+C_{p'}^2\,S_p^2\right)^2\,}{E_p+E_{p'}}
}$ .
\end{lemma}

\begin{proof} (a)\quad Corollary \ref{crl:properties} (a) and (d) yields
\begin{equation}\label{eq:bk}
B_k=C_k\,S_k\left( 1
-\gamma_{k\uparrow}^{\ast}\gamma_{k\uparrow}
-\gamma_{-k\downarrow}^{\ast}\gamma_{-k\downarrow} \right)
-C_k^2\,\gamma_{k\uparrow}\gamma_{-k\downarrow}
-S_k^2\,\gamma_{k\uparrow}^{\ast}\gamma_{-k\downarrow}^{\ast}\,.
\end{equation}
Hence,\begin{eqnarray*}
B_k\Psi_{BCS}&=&\left( C_k\,S_k-S_k^2\,\gamma_{k\uparrow}^{\ast}\gamma_{-k\downarrow}^{\ast} \right)\Psi_{BCS}\,,\\
B_k^{\ast}\Psi_{BCS}&=&\left( C_k\,S_k+C_k^2\,\gamma_{k\uparrow}^{\ast}\gamma_{-k\downarrow}^{\ast} \right)\Psi_{BCS}\,,\\
B_{k'}^{\ast}\,B_k\Psi_{BCS}
&=&\left( C_k\,S_k\,C_{k'}\,S_{k'}+C_k\,S_k\,C_{k'}^2\gamma_{k'\uparrow}^{\ast}
\gamma_{-k'\downarrow}^{\ast}-S_k^2\,C_{k'}\,S_{k'}\gamma_{k\uparrow}^{\ast}
\gamma_{-k\downarrow}^{\ast}\right. \\
& & \left. -S_k^2\,C_{k'}^2\gamma_{k\uparrow}^{\ast}\,
\gamma_{-k\downarrow}^{\ast}\,\gamma_{k'\uparrow}^{\ast}\,
\gamma_{-k'\downarrow}^{\ast} \right)\Psi_{BCS}\,.
\end{eqnarray*}
Part (a) thus follows. Parts (b) and (c) follow immediately from (a) and Corollary \ref{crl:properties} (a), (b).
\end{proof}

\begin{lemma}\label{lm:bkphi} {\rm (a)}
\begin{eqnarray*}
\hspace{-5mm}
B_k\Phi&=&\left( C_k\,S_k-S_k^2\,\gamma_{k\uparrow}^{\ast}\gamma_{-k\downarrow}^{\ast}\right)\Phi+C_k^2\sum_{p\in\Lambda} \frac{\,U_{k,\,p}
\left( C_k^2\,S_p^2+C_p^2\,S_k^2 \right)\,}{E_k+E_p}\,
\gamma_{p\uparrow}^{\ast}\,\gamma_{-p\downarrow}^{\ast}\Psi_{BCS}\\
& &-2C_k\,S_k \sum_{p\in\Lambda} \frac{\,U_{k,\,p}
\left( C_k^2\,S_p^2+C_p^2\,S_k^2 \right)\,}{E_k+E_p}\,
\gamma_{k\uparrow}^{\ast}\,\gamma_{-k\downarrow}^{\ast}
\gamma_{p\uparrow}^{\ast}\,\gamma_{-p\downarrow}^{\ast}\Psi_{BCS}\,.
\end{eqnarray*}

\noindent {\rm (b)}\quad $\displaystyle{
\left( \Phi,\,B_k\Phi \right)=C_k\,S_k\left\{ \left( \Phi,\,\Phi \right)
-2\sum_{p\in\Lambda}\frac{\,U_{k,\,p}^{\,2}
\left( C_k^2\,S_p^2+C_p^2\,S_k^2 \right)^2\,}{\left( E_k+E_p\right)^2} \right\} }$ .

\noindent {\rm (c)}
\begin{eqnarray*}
& &\left( \Phi,\,B_{k'}^{\ast}B_k\Phi \right)\\
&=&C_k\,S_k\,C_{k'}\,S_{k'}\left[ \left( \Phi,\,\Phi \right)
\phantom{ \frac{\,U_{k,\,p}^{\,2}\left( C_k^2\,S_p^2+C_p^2\,S_k^2 \right)^2\,}{\left( E_k+E_p\right)^2} } \right.\\
& &\qquad\quad-\left. 2\sum_{p\in\Lambda} \left\{ \frac{\,U_{k,\,p}^{\,2}
\left( C_k^2\,S_p^2+C_p^2\,S_k^2 \right)^2\,}{\left( E_k+E_p\right)^2}
+\frac{\,U_{k',\,p}^{\,2}
\left( C_{k'}^2\,S_p^2+C_p^2\,S_{k'}^2 \right)^2\,}{\left( E_{k'}+E_p\right)^2}
\right\} \right] \\
& &+4C_k\,S_k\,C_{k'}\,S_{k'}\frac{\,U_{k,\,k'}^{\,2}
\left( C_k^2\,S_{k'}^2+C_{k'}^2\,S_k^2 \right)^2\,}{\left( E_k+E_{k'}\right)^2}
\\
& &+\left( C_k^2\,C_{k'}^2+S_k^2\,S_{k'}^2 \right)\sum_{p\in\Lambda}
\frac{\, U_{k,\,p}U_{k',\,p}\left( C_k^2\,S_p^2+C_p^2\,S_k^2 \right)
\left( C_{k'}^2\,S_p^2+C_p^2\,S_{k'}^2 \right)\,}{
\left( E_k+E_p\right)\left( E_{k'}+E_p\right) }\,.
\end{eqnarray*}
\end{lemma}

\begin{proof} (a)\quad A straightforward calculation based on Corollary \ref{crl:properties} (a) gives
\begin{eqnarray*}
\hspace{-4mm}\left( \gamma_{k\uparrow}^{\ast}\gamma_{k\uparrow}
+\gamma_{-k\downarrow}^{\ast}\gamma_{-k\downarrow} \right)\Phi
&=&2\sum_{p\in\Lambda} \frac{\,U_{k,\,p}
\left( C_k^2\,S_p^2+C_p^2\,S_k^2 \right)\,}{E_k+E_p}\,
\gamma_{k\uparrow}^{\ast}\,\gamma_{-k\downarrow}^{\ast}
\gamma_{p\uparrow}^{\ast}\,\gamma_{-p\downarrow}^{\ast}\Psi_{BCS}\,,\\
\gamma_{k\uparrow}\gamma_{-k\downarrow}\Phi
&=&-\sum_{p\in\Lambda} \frac{\,U_{k,\,p}
\left( C_k^2\,S_p^2+C_p^2\,S_k^2 \right)\,}{E_k+E_p}\,
\gamma_{p\uparrow}^{\ast}\,\gamma_{-p\downarrow}^{\ast}\Psi_{BCS}\,.
\end{eqnarray*}
Part (a) thus follows from \eqref{eq:bk}. Parts (b) and (c) follow immediately from (a) and Corollary \ref{crl:properties} (a), (b).
\end{proof}

\begin{lemma}\label{lm:trianglee} Let $H'$ be as in \eqref{eq:hprime}. Then
\[
\left( \Phi,\,H'\Phi \right)=\left\{\, 1+\left( \Phi,\,\Phi \right)\,\right\}
\triangle E,
\]
where
\begin{eqnarray*}
\triangle E&=&\sum_{k,\,k'\in\Lambda} U_{k,\,k'}\,
\frac{\, C_k^2\,C_{k'}^2+S_k^2\,S_{k'}^2\,}{1+\left( \Phi,\,\Phi \right)}\,
\times \\
& &\times \sum_{p\in\Lambda}
\frac{\, U_{k,\,p}U_{k',\,p}\left( C_k^2\,S_p^2+C_p^2\,S_k^2 \right)
\left( C_{k'}^2\,S_p^2+C_p^2\,S_{k'}^2 \right)\,}{
\left( E_k+E_p\right)\left( E_{k'}+E_p\right) } \\
& &+4\sum_{k,\,k'\in\Lambda} U_{k,\,k'}\,
\frac{\, C_k\,S_k\,C_{k'}\,S_{k'}\,}{1+\left( \Phi,\,\Phi \right)}\,
\frac{\, U_{k,\,k'}^{\,2}
\left( C_k^2\,S_{k'}^2+C_{k'}^2\,S_k^2 \right)^2\,}
{\left( E_k+E_{k'}\right)^2}\,.
\end{eqnarray*}
\end{lemma}

\begin{proof}\quad By \eqref{eq:hprime},
\begin{eqnarray*}
\left( \Phi,\,H'\Phi \right)&=&\sum_{k,\,k'\in\Lambda} U_{k,\,k'} \left[\,
\left( \Phi,\, B_{k'}^{\ast}\,B_k \Phi \right)
-C_{k'}\,S_{k'}\left( \Phi,\,\left\{ B_k^{\ast}+B_k \right\}\Phi \right)
\right.\\
& &\quad +\left.C_k\,S_k\,C_{k'}\,S_{k'}\left( \Phi,\,\Phi \right) \,\right].
\end{eqnarray*}
The result thus follows from Lemma \ref{lm:bkphi} (b), (c).
\end{proof}

Note that $\triangle E<0$. We now show that the state $\Psi$ above has a lower energy than the BCS state $\Psi_{BCS}$ .

\begin{theorem} Let $\triangle E$ be as in Lemma \ref{lm:trianglee}. Then
the state $\Psi$ has a lower energy than the BCS state $\Psi_{BCS}$, and hence than the Fermi vacuum $\Psi_F$ ,i.e.,
\[
\left( \Psi,\, H\Psi \right)-\left( \Psi_{BCS},\, H\Psi_{BCS}\right)=\triangle E<0.
\]
\end{theorem}

\begin{proof}\quad Since $\displaystyle{\left( \Psi,\, H\Psi \right)=\left( \Psi,\, H_M\Psi \right)+\left( \Psi,\, H'\Psi \right)}$, \  the result follows from Lemma \ref{lm:ebcs} (c), Lemma \ref{lm:hprimeebcs} (b), (c) and Lemma
\ref{lm:trianglee}.
\end{proof}

\section{A new gap equation}

In section 3 we use the BCS state $\Psi_{BCS}$ to deal with the expectation
values of the operators $C_{-k\downarrow}C_{k\uparrow}$ and
$C_{k\uparrow}^{\ast}C_{-k\downarrow}^{\ast}$ (see \eqref{eq:vev}). But we originally need to use the ground state of the BCS Hamiltonian to deal
with the expectation values of such operators. The ground state of the BCS Hamiltonian is studied by several authors. See Mattis and Lieb \cite{mattislieb}, Richardson \cite{richardson} and von Delft \cite{delft} for example. They assumed that $U_{k,\,k'}$ (see \eqref{eq:hamiltonian}) is a negative constant if $k$ and $k'$ both belong to the neighborhood of the Fermi surface, and 0 otherwise. So little is known about the ground state when $U_{k,\,k'}$ does \textit{not} satisfy the assumption just above. We therefore try to use our superconducting state $\Psi$ in the preceding section instead (see \eqref{eq:tildevev} below). This is because our state $\Psi$ has a lower energy than the BCS state $\Psi_{BCS}$.

To this end we have to begin with the following new gap equation. Let
$\widetilde{\Delta}_k$ be a function of $k\in\Lambda$. We assume the existence
of the following $\widetilde{\Delta}_k$ : \quad $\widetilde{\Delta}_k$
satisfies $\widetilde{\Delta}_k \geq 0$ and $\widetilde{\Delta}_{-k}=\widetilde{\Delta}_k$, and is also a solution to the new gap equation
\begin{equation}\label{eq:tildegapequation}
\widetilde{\Delta}_k=-\frac{1}{\,2\,} \sum_{k'\in\Lambda}U_{k,\,k'}\,
\frac{\widetilde{\Delta}_{k'}}
{\,\sqrt{\,\xi_{k'}^2+\widetilde{\Delta}_{k'}^2\,}\,}
\left( 1-\frac{4D_{k'}}{\,D+2\,}\right)\,,
\end{equation}
where
\[
D_{k'}=\frac{1}{\,4\,} \sum_{p\in\Lambda} \frac{\,U_{k',\,p}^{\,2}\,}
{\,\left( \sqrt{\,\xi_{k'}^2+\widetilde{\Delta}_{k'}^2\,}+
\sqrt{\,\xi_p^2+\widetilde{\Delta}_p^2\,} \right)^2\,}
\left( 1-\frac{ \xi_{k'}\,\xi_p }{\,
\sqrt{\,\xi_{k'}^2+\widetilde{\Delta}_{k'}^2\,}
\sqrt{\,\xi_p^2+\widetilde{\Delta}_p^2\,} \,} \right)^2,
\]
$\displaystyle{ D=\sum_{k'\in\Lambda} D_{k'} }$ .

\begin{remark}\label{rmk:numerical}
A numerical calculation gives $4D_{k'}/(D+2)\leq O(10^{-17})$ in the case of aluminum. So it is expected that $\widetilde{\Delta}_k$ is nearly equal to
$\Delta_k$ and that $\widetilde{\Delta}_k \geq 0$.
\end{remark}

Let $\widetilde{\theta}_k$ be a function of $k\in\Lambda$ and let it satisfy
\begin{equation}\label{eq:tildetheta}
\sin 2\widetilde{\theta}_k=\frac{\widetilde{\Delta}_k}{\,\sqrt{\,\xi_k^2+
\widetilde{\Delta}_k^2\,}\,}\,,\qquad
\cos 2\widetilde{\theta}_k=\frac{\xi_k}{\,\sqrt{\,\xi_k^2+
\widetilde{\Delta}_k^2\,}\,}
\end{equation}
with $0 \leq \widetilde{\theta}_k \leq \pi/2$. \  Note that
$\widetilde{\theta}_{-k}=\widetilde{\theta}_k$ . We denote by
$\widetilde{G}_B$ the following bounded, selfadjoint operator on
$\mathcal{H}=\mathbb{C}^{2^{2M}}$:
\[
\widetilde{G}_B=i\sum_{k\in\Lambda} \widetilde{\theta}_k
\left( C_{-k\downarrow}C_{k\uparrow}-
C_{k\uparrow}^{\ast}C_{-k\downarrow}^{\ast} \right).
\]
Here, $\widetilde{\theta}_k$ is given by \eqref{eq:tildetheta}. We set
\[
\widetilde{\Psi}_{BCS}=e^{i\widetilde{G}_B}|0\rangle
\in\mathcal{H}=\mathbb{C}^{2^{2M}}.
\]
We now give another expression for the state $\widetilde{\Psi}_{BCS}$. A proof
similar to that of Lemma \ref{lm:bcsstate} yields the following.

\begin{lemma}\label{lm:tildebcsstate} \quad $\displaystyle{
\widetilde{\Psi}_{BCS}=\left\{ \,\prod_{k\in\Lambda}\left(
\cos\widetilde{\theta}_k+\sin\widetilde{\theta}_k\,
C_{k\uparrow}^{\ast}C_{-k\downarrow}^{\ast} \right) \,\right\}|0\rangle }$.
\end{lemma}

An immediate implication of this lemma is the following.

\begin{corollary}\label{crl:tildevevbk}
\[
\left( \widetilde{\Psi}_{BCS},\,C_{-k\downarrow}C_{k\uparrow}\widetilde{\Psi}_{BCS}\right)
=\left( \widetilde{\Psi}_{BCS},\,C_{k\uparrow}^{\ast}C_{-k\downarrow}^{\ast}\widetilde{\Psi}_{BCS}\right)
=\frac{1}{\,2\,}\sin 2\widetilde{\theta}_k\,.
\]
\end{corollary}

We introduce another Bogoliubov transformation of $C_{k\,\sigma}$:
\begin{equation}\label{eq:tildegamma}
\widetilde{\gamma}_{k\,\sigma}=e^{i\widetilde{G}_B}C_{k\,\sigma}
e^{-i\widetilde{G}_B}.
\end{equation}
Note that the operator $\widetilde{\gamma}_{k\,\sigma}$ and its adjoint
operator $\widetilde{\gamma}_{k\,\sigma}^{\ast}$ are both bounded linear
operators on $\mathcal{H}=\mathbb{C}^{2^{2M}}$.

A proof similar to that of Corollary \ref{crl:properties} yields the following.

\begin{corollary}\label{crl:tildeproperties} The operators
$\widetilde{\gamma}_{k\,\sigma}$ and $\widetilde{\gamma}_{k\,\sigma}^{\ast}$
satisfy the following.

\noindent {\rm (a)} \quad $\left\{ \widetilde{\gamma}_{k\sigma},\,
\widetilde{\gamma}_{k'\sigma'}^{\ast} \right\}=
\delta_{k k'}\delta_{\sigma \sigma'},\qquad
\left\{ \widetilde{\gamma}_{k\sigma},\,\widetilde{\gamma}_{k'\sigma'} \right\}
=\left\{ \widetilde{\gamma}_{k\sigma}^{\ast},\,
\widetilde{\gamma}_{k'\sigma'}^{\ast} \right\}=0$.

\noindent {\rm (b)}\quad $\widetilde{\gamma}_{k\,\sigma}\widetilde{\Psi}_{BCS}
=0$\quad for each $k\in\Lambda$ and for each
$\sigma=\uparrow,\,\downarrow$ .

\noindent {\rm (c)}\quad $\displaystyle{
\left\{ \begin{array}{ll} \displaystyle{
\widetilde{\gamma}_{k\uparrow}=\cos\widetilde{\theta}_k\,C_{k\uparrow}
-\sin\widetilde{\theta}_k\, C_{-k\downarrow}^{\ast}\,,}
& \\ \noalign{\vskip0.3cm}
\displaystyle{
\widetilde{\gamma}_{-k\downarrow}=\sin\widetilde{\theta}_k\,
C_{k\uparrow}^{\ast}+\cos\widetilde{\theta}_k\,C_{-k\downarrow}\,.} &
\end{array}
\right.
}$

\noindent {\rm (d)}\quad $\displaystyle{
\left\{ \begin{array}{ll} \displaystyle{
C_{k\uparrow}=\cos\widetilde{\theta}_k\,\widetilde{\gamma}_{k\uparrow}
+\sin\widetilde{\theta}_k\,\widetilde{\gamma}_{-k\downarrow}^{\ast}\,,}
& \\ \noalign{\vskip0.3cm}
\displaystyle{
C_{-k\downarrow}=-\sin\widetilde{\theta}_k\,
\widetilde{\gamma}_{k\uparrow}^{\ast}+\cos\widetilde{\theta}_k\,
\widetilde{\gamma}_{-k\downarrow}\,.} &
\end{array}
\right.
}$
\end{corollary}

Set $\widetilde{E}_k=\sqrt{\,\xi_k^2+\widetilde{\Delta}_k^2\,}$, $k\in\Lambda$
and set $B_k=C_{-k\downarrow}C_{k\uparrow}$. We abbreviate
$\sin\widetilde{\theta}_k$ (resp. $\cos\widetilde{\theta}_k$)
to $\widetilde{S}_k$ (resp. to $\widetilde{C}_k$). We now consider the
following vector in $\mathcal{H}=\mathbb{C}^{2^{2M}}$:
\begin{equation}\label{eq:tildepsi}
\widetilde{\Psi}=\frac{\widetilde{\Psi}_{BCS}+\widetilde{\Phi}}
{\,\sqrt{\, 1+\left( \widetilde{\Phi},\,\widetilde{\Phi} \right) \,}\,},
\end{equation}
where \quad $\displaystyle{
\widetilde{\Phi}=\frac{1}{\,2\,}\sum_{p,\,p'\in\Lambda} \frac{\,U_{p,\,p'}
\left( \widetilde{C}_p^2\,\widetilde{S}_{p'}^2+
\widetilde{C}_{p'}^2\,\widetilde{S}_p^2\right)\,}
{\widetilde{E}_p+\widetilde{E}_{p'}}\,
\widetilde{\gamma}_{p\uparrow}^{\ast}\,
\widetilde{\gamma}_{-p\downarrow}^{\ast}\,
\widetilde{\gamma}_{p'\uparrow}^{\ast}\,
\widetilde{\gamma}_{-p'\downarrow}^{\ast}\,\widetilde{\Psi}_{BCS} }$ . For all
$k\in\Lambda$, set
\begin{equation}\label{eq:tildevev}
\left\{ \begin{array}{ll} \displaystyle{
C_{-k\downarrow}C_{k\uparrow}=
\left(\widetilde{\Psi},\, C_{-k\downarrow}C_{k\uparrow}\widetilde{\Psi}\right)
+\widetilde{b}_k \,,} & \\
\noalign{ \vskip0.3cm }
\displaystyle{
C_{k\uparrow}^{\ast}C_{-k\downarrow}^{\ast}=
\left(\widetilde{\Psi},\,C_{k\uparrow}^{\ast}C_{-k\downarrow}^{\ast}
\widetilde{\Psi}\right)+\widetilde{b}_k^{\ast}
\,.} &
\end{array}\right.
\end{equation}
Here, $\widetilde{\Psi}$ is given by \eqref{eq:tildepsi}. The following corresponds to Corollary \ref{crl:bcsexpectationvalue} (b).

\begin{corollary}\label{crl:tildepsi}\quad $\displaystyle{
\widetilde{\Delta}_k=-\sum_{k'\in\Lambda} U_{k,\,k'}\,
\left( \widetilde{\Psi},\,C_{-k'\downarrow}C_{k'\uparrow}\widetilde{\Psi}
\right) }$.
\end{corollary}

\begin{proof}\quad By \eqref{eq:tildepsi},
\begin{eqnarray*}
& &\left( \widetilde{\Psi},\,B_k\widetilde{\Psi}\right)\\
&=&\frac{\,\left( \widetilde{\Psi}_{BCS},\,B_k\widetilde{\Psi}_{BCS}\right)+
\left( \widetilde{\Psi}_{BCS},\,B_k\widetilde{\Phi}\right)+
\left( \widetilde{\Phi},\,B_k\widetilde{\Psi}_{BCS}\right)+
\left( \widetilde{\Phi},\,B_k\widetilde{\Phi}\right) \,}
{1+\left( \widetilde{\Phi},\,\widetilde{\Phi} \right)}.
\end{eqnarray*}
An argument similar to that in the proof of Lemma \ref{lm:bkphi} (a) gives
\begin{eqnarray*}
\hspace{-4mm}
B_k\widetilde{\Phi}&=&
(\widetilde{C}_k\,\widetilde{S}_k
-\widetilde{S}_k^2\,\widetilde{\gamma}_{k\uparrow}^{\ast}\widetilde{\gamma}_{-k\downarrow}^{\ast} ) \widetilde{\Phi}
+\widetilde{C}_k^2\sum_{p\in\Lambda} \frac{\,U_{k,\,p}
\left( \widetilde{C}_k^2\,\widetilde{S}_p^2+
 \widetilde{C}_p^2\,\widetilde{S}_k^2 \right)\,}
{\widetilde{E}_k+\widetilde{E}_p}\,
\widetilde{\gamma}_{p\uparrow}^{\ast}\,\widetilde{\gamma}_{-p\downarrow}^{\ast}
\widetilde{\Psi}_{BCS}\\
& &-2\widetilde{C}_k\,\widetilde{S}_k \sum_{p\in\Lambda} \frac{\,U_{k,\,p}
\left( \widetilde{C}_k^2\,\widetilde{S}_p^2
 +\widetilde{C}_p^2\,\widetilde{S}_k^2 \right)\,}
{\widetilde{E}_k+\widetilde{E}_p}\,
\widetilde{\gamma}_{k\uparrow}^{\ast}\,\widetilde{\gamma}_{-k\downarrow}^{\ast}
\widetilde{\gamma}_{p\uparrow}^{\ast}\,\widetilde{\gamma}_{-p\downarrow}^{\ast}
\widetilde{\Psi}_{BCS}\,.
\end{eqnarray*}
Hence, $\displaystyle{\left( \widetilde{\Psi}_{BCS},\,B_k\widetilde{\Phi}\right)=0}$,
\[
\left( \widetilde{\Phi},\,B_k\widetilde{\Phi} \right)
=\widetilde{C}_k\,\widetilde{S}_k\left\{
\left( \widetilde{\Phi},\,\widetilde{\Phi} \right)
-2\sum_{p\in\Lambda} \frac{\,U_{k,\,p}^2
\left( \widetilde{C}_k^2\,\widetilde{S}_p^2
 +\widetilde{C}_p^2\,\widetilde{S}_k^2 \right)^2\,}
{\left( \widetilde{E}_k+\widetilde{E}_p\right)^2} \right\}.
\]
Similarly, $\displaystyle{ \left( \widetilde{\Phi},\,B_k\widetilde{\Psi}_{BCS}
\right)=0}$. The result thus follows from Corollary \ref{crl:tildevevbk} and
\eqref{eq:tildegapequation}.
\end{proof}

Combining Corollary \ref{crl:tildepsi} with \eqref{eq:tildevev} yields the
following.

\begin{lemma} Set
\begin{eqnarray*}
\widetilde{H}_M&=&\sum_{k\in\Lambda,\,\sigma=\uparrow,\,\downarrow}\xi_k\, 
C_{k\,\sigma}^{\ast}C_{k\,\sigma}-\sum_{k\in\Lambda} \widetilde{\Delta}_k
\left( C_{-k\downarrow}C_{k\uparrow}+
C_{k\uparrow}^{\ast}C_{-k\downarrow}^{\ast} \right)\\
& &\quad +\sum_{k\in\Lambda} \widetilde{\Delta}_k
\left(\widetilde{\Psi},\,C_{-k\downarrow}C_{k\uparrow}\widetilde{\Psi}\right).
\end{eqnarray*}
Then the BCS Hamiltonian \eqref{eq:hamiltonian} is rewritten as
$\displaystyle{
H=\widetilde{H}_M+\sum_{k,\,k'\in\Lambda} U_{k,\,k'}\,
\widetilde{b}_{k'}^{\ast} \,\widetilde{b}_k}$ .
\end{lemma}

\begin{remark}
The Hamiltonian $\widetilde{H}_M$ as well as $H_M$ (see Lemma \ref{lm:h-hm}) is also the mean field approximation for the BCS Hamiltonian $H$.
\end{remark}

A proof similar to that of Proposition \ref{prp:hm} yields the following.

\begin{proposition}\label{prp:tildehm}
\begin{eqnarray*}
\widetilde{H}_M&=&\sum_{k\in\Lambda,\,\sigma=\uparrow,\,\downarrow}
\sqrt{\,\xi_k^2+\widetilde{\Delta}_k^2\,}\,
\widetilde{\gamma}_{k\,\sigma}^{\ast}\widetilde{\gamma}_{k\,\sigma}\\
& &\quad +\sum_{k\in\Lambda}\left\{ \xi_k-\sqrt{\,\xi_k^2+\widetilde{\Delta}_k^2\,}+\widetilde{\Delta}_k \left( \widetilde{\Psi},\,C_{-k\downarrow}C_{k\uparrow}\widetilde{\Psi} \right) \right\}.
\end{eqnarray*}
\end{proposition}

Proposition \ref{prp:tildehm}, together with Corollary
\ref{crl:tildeproperties} (b), immediately yields the following.

\begin{corollary}\label{crl:tildeground} {\rm (a)}\quad The state
$\widetilde{\Psi}_{BCS}$ is the ground state of $\widetilde{H}_M$, and
the ground state energy $\widetilde{E}_{BCS}$ is given by
\[
\widetilde{E}_{BCS}=\sum_{k\in\Lambda} \left\{
\xi_k-\sqrt{\,\xi_k^2+\widetilde{\Delta}_k^2\,}+
\widetilde{\Delta}_k \left( \widetilde{\Psi},\,C_{-k\downarrow}C_{k\uparrow}
\widetilde{\Psi} \right) \right\}.
\]

\noindent {\rm (b)}\quad Let $\widetilde{E}_{BCS}$ be as in (a). Then the
spectrum of $\widetilde{H}_M$ is given by
\[
\sigma\left( \widetilde{H}_M\right)=\left\{ \sum_{k\in\Lambda}
\sqrt{\xi_k^2+\widetilde{\Delta}_k^2}\,
\left(\,N_{k\uparrow}+N_{k\downarrow}\,\right)+
\widetilde{E}_{BCS} \right\}_{N_{k\uparrow},\, N_{k\downarrow}=0,\,1}.
\]
\end{corollary}

\begin{remark}
We see from Corollary \ref{crl:tildeground} (b) that it takes a finite energy
$\sqrt{\xi_k^2+\widetilde{\Delta}_k^2}$ ($>\widetilde{\Delta}_k$) to excite
a particle from the state $\widetilde{\Psi}_{BCS}$ to an upper energy state. So
$\widetilde{\Delta}_k$ as well as $\Delta_k$ corresponds exactly to the
energy gap, and hence $\widetilde{\Delta}_k$ as well as $\Delta_k$ is the gap
function.
\end{remark}

\begin{remark}
Beginning with our new gap equation \eqref{eq:tildegapequation} we arrive at the results similar to those in section 3, i.e., to those in the BCS-Bogoliubov theory.
\end{remark}

\end{document}